\documentclass{amsart}

\usepackage{amsopn, amsthm, amsgen, amscd,amsmath, amssymb}
\usepackage{amsaddr}
\usepackage{graphicx}
\usepackage{subfigure}
\usepackage[nodots]{numcompress}
\usepackage{verbatim}
\usepackage{fancybox}
\usepackage{subfigure}
\usepackage[usenames,dvipsnames]{color}
\usepackage{sidecap}
\usepackage{caption}
\usepackage{ifthen}
\usepackage{textcomp}
\usepackage[all]{xy}

\newcommand{\Z}{{\mathbb Z}}

\newcommand{\R}{{\mathbb R}}
\newcommand{\C}{{\mathbb C}}
\newcommand{\Q}{{\mathbb Q}}

\theoremstyle{plain}
\newtheorem{teo}{Theorem}[section]
\newtheorem{lema}[teo]{Lemma}
\newtheorem{cor}[teo]{Corollary}
\newtheorem{prop}[teo]{Proposition}

\theoremstyle{definition}
\newtheorem{defi}{Definition}[section]

\theoremstyle{remark}
\newtheorem{remark}{Remark}[section]

\title[]{ON THE KHOVANOV HOMOLOGY OF SURGERIES}

\author{J.M.Burgos}
\address{Center for Research and Advanced Studies of the National Polytechnic Institute,\\Mathematics Department, CINVESTAV, Mexico City, Mexico, 07360.}
\email{burgos@math.cinvestav.mx}

\begin{document}
\vspace{2cm}

\maketitle

\begin{abstract}
We show a spectral sequence for the rational Khovanov homology of an oriented link in terms of the rational Khovanov complexes and homologies of the link surgeries along an admissible cut. As a non trivial corollary, we give an explicit splitting formula for the Jones polynomial.
\end{abstract}



\section{Introduction}

The ``divide and conquer'' \footnote{``divide et impera'', Julius Caesar.} metaprinciple has been one of the most effective strategies ever played and mathematics is not an exception. However, to be able to apply this principle it is neccessary to:
\begin{enumerate}
\item \textit{Decompose} the problem into subproblems.
\item \textit{Solve} the subproblems.
\item \textit{Assemble} the partial solutions into a a solution of the main original problem.
\end{enumerate}
Some problems are local in nature and they have the Decompose/Assemble parts as god given. Some other's not. For these non local problems the Decompose/Assemble parts live in a symbiotic relationship \footnote{Assuming that every subproblem is solvable, the design of either of the Decompose/Assemble parts must contemplate the limitations of the other.} and usually their design is an art more than a science. It is important to remark that the design of these parts need theory and cannot be solved by brute force. For example, we could have a supercomputer computing a thousand of subproblems simultaneously but if we have neither a formula nor an algorithm to assemble these partial solutions into a global one, it is pointless \footnote{In spite of the example, this paper is not about computing nor parallel processing neither computation complexity.}. The Decompose/Assemble parts are real mathematical problems and usually non trivial ones.

The Khovanov homology \cite{Khovanov_original} is a great example of a non local problem. It is the categorification of the Jones polynomial and it proves to be a finer invariant for non alternating knots \footnote{For example, it distinguishes $9_{42}$ from its mirror image whereas the Jones polynomial doesn't. In \cite{classicBN}, section 4.5, there is a complete list of pairs of prime knots up to 11 crossings whose Jones polynomials are equal but their rational Khovanov homology are not.}. As a non local problem, a great effort has been made in the design of its Decompose/Assemble parts.

Following this direction, among other things the remarkable paper \cite{Tangles_BN} extends Khovanov homology theory for tangles and defines the TQFT \textit{tautological functor} with the remarkable property that it is a planar algebra morphism. However, the local theory obtained is not the usual Khovanov theory and to specialize the local theory to the usual one, other TQFT functor is needed. Unfortunatley, this other functor loose information in such a way that its local behavior is lost in the sense that it is no longer a planar algebra morphism and there is no apparent way to get neither a splitting formula nor a spectral sequence for the usual Khovanov homology\footnote{The same happen with Lee's homology or the characteristic two theory.}. Nevertheless, within local Khovanov theory, the tools of \textit{delooping} and \textit{Gaussian elimination} developed in \cite{Fast_Khovanov_BN} enables dramatically powerful Khovanov homology calculations.

Another example is the paper \cite{open_closed_TQFT} where an open-closed TQFT is used in order to get a theory for tangles. Relating this theory to the Khovanov homology faces the same problem as above. However, they give a spectral sequence converging to the Khovanov homology in positive characteristic.


We say that a Jordan curve $C$ is an admissible cut of an oriented link $L$ if it is transversal to $L$ and walking along the curve in some direction the orientation of the $2n$ intersection points alternate. The cut $C$ separates the link diagram $L$ in two tangles $T_{1}$ and $T_{2}$ and all of the possible non crossing closures of these are the surgeries $L_{1}^{\mathcal{A}}$ and $L_{2}^{\mathcal{B}}$ respectively of the link diagram $L$, indexed by non crossing partitions $\mathcal{A}$ and $\mathcal{B}$. The following are the main results of the paper:

\begin{teo}
Given an admissible cut $C$ of an oriented link diagram $L$, there is a double chain complex whose rows and columns are linear combinations of the Khovanov complexes of the surgeries $L_{1}^{\mathcal{A}}$ and $L_{2}^{\mathcal{B}}$ respectively such that its total complex is homotopic to the Khovanov complex of $L$.
\end{teo}

\begin{teo}
In characteristic zero, given an admissible cut $C$ of an oriented link diagram $L$ there is spectral sequence calculated from the Khovanov homology of the surgeries $L_{2}^{\mathcal{B}}$ and the Khovanov complexes of the surgeries $L_{1}^{\mathcal{A}}$, converging to the Khovanov homology of $L$.
\end{teo}

As a non trivial corollary of the proof, we have the following splitting formula for the Jones polynomial respect to the Khovanov parameter:

\begin{cor}
Given an admissible cut $C$ of an oriented link diagram $L$, there is a matrix $b(q)$ with entries in the field of rational functions $k(q)$ such that:
$$J(L) = \sum_{\mathcal{A},\mathcal{B}\in NC_{n}} J\left( L_{1}^{\mathcal{A}}\right)\ b_{\mathcal{A}\mathcal{B}}(q)\ J\left( L_{2}^{\mathcal{B}}\right)$$
\end{cor}

\section{Khovanov homology in a nutshell}\label{Khovanov_Nutshell}

This section follows closely \cite{classicBN} \footnote{Khovanov original paper works over the ring $\Z[c]$ such that $\deg c=2$.}. For every $q$-graded vector space $W:=\bigoplus_{m} W_{m}$ with homogeneous component $W_{m}$ of $q$-degree $m$, we define its $q$-dimension as the following Laurent polynomial:
$$\textsl{q}\dim W:= \sum_{m}q^{m}\dim W_{m}$$
We also define its \textit{q degree shift} $W\{l\}:= \bigoplus_{m} W\{l\}_{m}$ such that $W\{l\}_{m}:=W_{m-l}$. In particular, we have the formula:
$$\textsl{q}\dim W\{l\}= q^{l}\ \textsl{q}\dim W$$
If $\mathcal{C}= \ldots \mathcal{C}^{r}\xrightarrow{d^{r}} \mathcal{C}^{r+1}\ldots$ is a chain complex, possibly of graded vector spaces, we denote by $\mathcal{C}[s]$ the chain complex such that $\mathcal{C}[s]^{m}= \mathcal{C}^{m-s}$ with all differentials shifted accordingly.

From now on we will consider the $q$-graded vector space:
$$V:=V_{-1}\oplus V_{1}= \left\langle v_{-}\right\rangle\oplus \left\langle v_{+}\right\rangle$$
such that $\deg_{\textsl{q}}v_{\pm}:=\pm 1$. Its $q$-dimension is simply $\textsl{q}\dim V= q+q^{-1}$.

Given an oriented link diagram $L$ consider its set of crossings $\chi$ and the respective set of states $\{0,1\}^{\chi}$. Every state $\alpha$ defines an $\alpha$-smoothing of the link $L$ following the rule of Fig. \ref{Smoothing}. Denote this smoothing by $\mathcal{S}_{\alpha}(L)$. Every smoothing is a set of disjoint circles. For every state $\alpha$ define $k(\alpha)$ as the number of circles in the respective $\alpha$-smoothing and $r(\alpha)$ as the following number:
$$r(\alpha):= \sum_{c\ \textsl{is\ a\ crossing}}\ \alpha(c)$$

Consider the set of arrows as the subset of $\{0,1,*\}^{\chi}$ such that the symbol $*$ appears only once. Evaluating the symbol $*$ on $0$ and $1$ gives the source and the target of the arrow respectively. Then, every arrow can be seen as a cobordism of circles:
$$\mathcal{S}_{a}(L):\mathcal{S}_{a(0)}(L)\rightarrow \mathcal{S}_{a(1)}(L)$$
The set of states and arrows define the cobordism category $\mathcal{S}(L)$.

\begin{figure}
\begin{center}
  \includegraphics[width=0.7\textwidth]{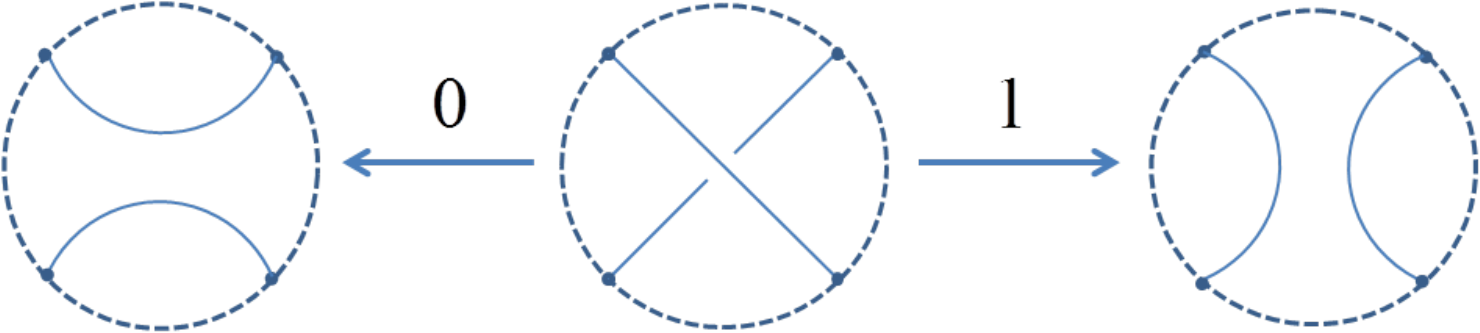}\\
  \end{center}
  \caption{Smoothing of a crossing.}\label{Smoothing}
\end{figure}

For every state $\alpha$ consider a total ordering in the respective $\mathcal{S}_{\alpha}(L)$ smoothing. Define the functor $F$ from the cobordism category $\mathcal{S}(L)$ to the category of graded vector spaces such that:
$$F\left(\mathcal{S}_{\alpha}(L)\right):= \bigotimes_{\substack{ a\in \mathcal{S}_{\alpha}(L) \\ a\ \textsl{is\ a\ circle}}} V_{a}$$
where $V_{a}$ is the $q$-graded vector space $V$ labeled by the circle $a$ in the smoothing.
The evaluation of a cobordism is a Frobenius algebra morphism\footnote{A Frobenius algebra is a finite dimensional associative algebra $A$ over a field $K$ with unit $1_{A}$ equipped with a linear functional $\varepsilon:A\rightarrow K$ such that the Kernel of $\varepsilon$ contains no nonzero left ideal of the algebra $A$. The functional $\varepsilon$ is called the \textit{counit}.} such that a pants decomposition of the cobordism evaluates as in Fig. \ref{Frobenius_Figure}. This functor $F$ is a $2D$\textit{-Topological Quantum Field Theory} according to the Moore-Segal axioms\footnote{These axioms are equivalent to the Atiyah axioms.}.

\begin{figure}
\begin{center}
  \includegraphics[width=0.7\textwidth]{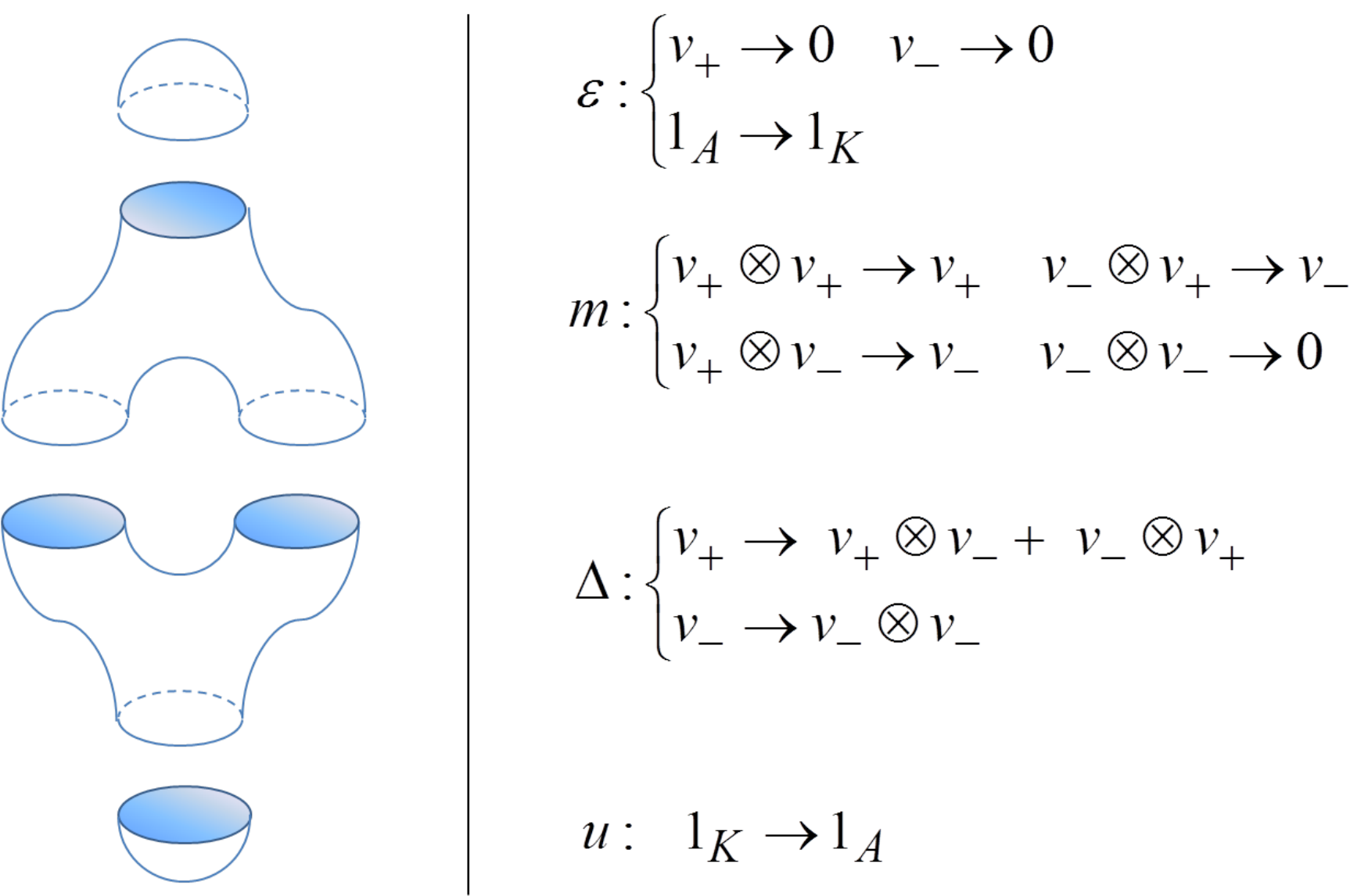}\\
  \end{center}
  \caption{A $2D$-Topological Quantum Field Theory.}\label{Frobenius_Figure}
\end{figure}

We define the chain complex $[\![L]\!]$ as follows: For every level $r$ define:
$$[\![L]\!]^{r}:=\bigoplus_{r(\alpha)=r}\ F\left(\mathcal{S}_{\alpha}(L)\right)\{r\}$$
$$d^{r}:= \bigoplus_{r(a(0))=r}\ (-1)^{sg(a)}d_{a}$$
where $d_{a}$ is the Frobenius morphism obtained by the functor $F$ evaluated on the cobordism $\mathcal{S}_{a}(L)$ in the respective degree shiftings. The sign $sg(a)$ of the arrow $a$ is defined as follows: Consider a total ordering in the set of crossings\footnote{Once we have a total order in the set of crossings, we see each state as a binary number and we order the sum $\bigoplus_{r(\alpha)=r}$ in the definition of $[\![L]\!]^{r}$ accordingly.} $\chi$ and denote by $j(*,a)$ the crossing such that the arrow $a$ evaluates to the symbol $*$. Define:
$$sg(a):=\sum_{j < j(*,a)}\ a(j)$$

Among other topological field theories, the functor $F$ was chosen in such a way that the Frobenius morphisms $d_{a}$ is a degree zero morphism \footnote{As an example of a topological field theory that doesn't have this property, consider the Eun Soo Lee functor \cite{Lee}.}. In particular, this grading is preserved in cohomology.

\begin{figure}
\begin{center}
  \includegraphics[width=0.4\textwidth]{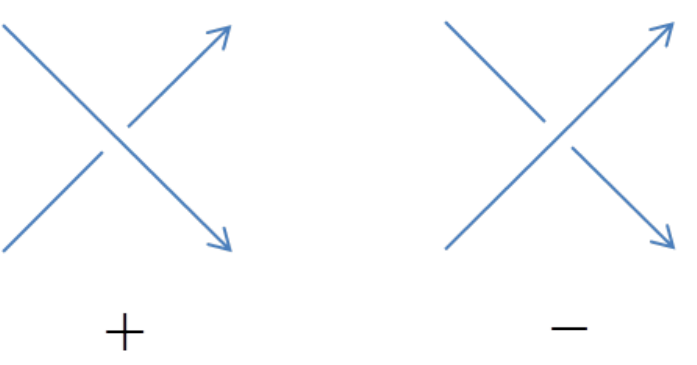}\\
  \end{center}
  \caption{Sign of a crossing.}\label{Sign_Crossing}
\end{figure}

We define the sign of a crossing as in Fig. \ref{Sign_Crossing}. Consider the number of positive and negative crossings $l^{+}$ and $l^{+}$ of the oriented link diagram $L$. The Khovanov complex is defined as follows:
$$\mathcal{C}(L):=[\![L]\!][-l^{-}]\{l^{+}-2l^{-}\}$$
The Khovanov complex is the categorification of the Jones polynomial in the sense that its grading Euler characteristic recovers this link invariant:
$$J(L)(q)= (q+q^{-1})^{-1}\ \chi_{\textsl{q}}(\mathcal{C}(L))$$
where $\chi_{q}$ is the $q$-graded Euler characteristic. The Khovanov parameter $q$ is related to the usual Jones parameter $t$ by the relation $q= -t^{1/2}$.

Up to homotopy, the Khovanov complex is a link invariant \cite{Khovanov_original}, \cite{classicBN}:

\begin{teo}\label{Khovanov_Homotopy}
The complex $\mathcal{C}(L)$ is a chain complex such that up to homotopy, it is defined on the ambient isotopic classes of oriented link diagrams; i.e. Different orderings in the set of crossings, set of smoothings and Reidemeister moves on the link diagram give homotopic chain complexes.
\end{teo}

In paticular, the \textit{Khovanov homology}\footnote{As the reader may have noticed, it is actually a cohomology.} is a link invariant:
$$Kh_{K}(L):=H\left(\mathcal{C}(L)\right)$$
The Poincar\'e polynomial of the Khovanov homology is called the \textit{Khovanov polynomial} of the link. In particular, the characteristic zero Khovanov polynomial recovers the Jones polynomial:
$$J(L)(q)=(q+q^{-1})^{-1}\ Kh_{\Q}(L)(-1,q)$$

\section{Admissible cuts and surgeries}

We define the notion of non crossing partitions as follows: Consider the k-th character $g_{k}:\R\rightarrow \C$ such that $g_{k}(t)=\exp(2\pi i t/k)$. For every partition $\mathcal{A}\in \Gamma_{n}$ such that $\mathcal{A}= \{m_{1},\ m_{2},\ldots m_{l}\}$ define:
$$Convex(\mathcal{A})=\{Convex\left(g_{n}(m_{1})\right),\ Convex\left(g_{n}(m_{2})\right),\ldots Convex\left(g_{n}(m_{l})\right)\}$$
where $Convex(S)$ denotes the convex hull of the set $S$ in $\C$. A partition $\mathcal{A}\in \Gamma_{n}$ will be called \textit{non crossing} \footnote{Non crossing partitions were introduced by Kreweras in \cite{Kreweras} and since then they have been widely used in different branches of mathematics \cite{mccammond}.} if for every pair of distinct convex sets in $Convex(\mathcal{A})$ their intersection is empty. The subset of non crossing partitions will be denoted by $NC_{n}$.

The permutation group $S_{n}$ acts on the set of $n$-partitions $\Gamma_{n}$ as follows: For every $n$-partition $\mathcal{A}= \{m_{1},\ m_{2},\ldots m_{l}\}$ we define:
$$\sigma\cdot\mathcal{A}:= \{\sigma(m_{1}),\ \sigma(m_{2}),\ldots \sigma(m_{l})\}$$
for every permutation $\sigma\in S_{n}$. See that the non crossing permutations $NC_{n}$ is closed under the action of the Dihedral group $D_{n}$ and this is the maximal permutation subgroup with this property.

\begin{defi}
An admissible cut of an oriented link diagram $L$ is a Jordan curve $C$ of the plane such that $C$ is transversal to the link diagram $L$ \footnote{In particular, the Jordan curve doesn't intersect any crossing of the link diagram $L$.} and, giving $C$ some orientation and walking the curve $C$ along this orientation, the intersection points orientation alternate.
\end{defi}

\begin{figure}
\begin{center}
  \includegraphics[width=0.4\textwidth, height=0.23\textheight]{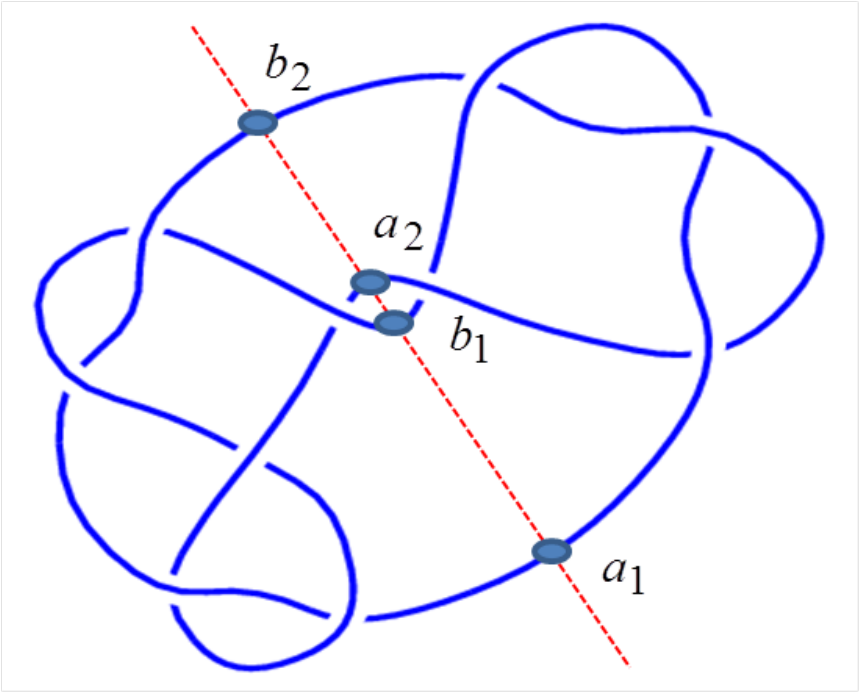}\\
  \end{center}
  \caption{Admissible cut of the knot $\overline{10_{25}}$.}\label{Admissible_Cut}
\end{figure}


Consider an oriented admissible cut $C$ with a marked point $c\in C$. Because of degree theory, there must be $2n$ intersection point with $L$, half of them positively oriented and the other half negatively oriented. Denote these points by:
$$a_{1},b_{1},a_{2},b_{2},\ldots a_{n},b_{n}$$
as we go through the curve $C$ along its orientation starting in the marking $c$. We choose the marking in such a way that $a_{1}$ is positively oriented\footnote{This is achieved just moving the marking along the admissible cut.}.

Consider the admissible cut $C$ in the one point compactification of the plane. The Jordan curve $C$ separates the sphere $S^{2}$ in two regions $R1$ and $R2$, each one diffeomorphic to the disk $\Delta$ by the diffeomorphisms $\varphi_{1}$ and $\varphi_{2}$ respectively. One of them preserves the orientation and the other one reverses it. We choose the regions in such a way that $\varphi_{2}$ preserves the orientation. These diffeomorphisms have unique extensions to the boundary of $R1$ and $R2$ respectively and we will denote these extensions with the same name. We choose these diffeomorphisms in such a way that their extensions map the intersection points:
$$a_{1},b_{1},\ldots a_{n},b_{n}$$
to the points in the circle:
$$g_{2n}(1),g_{2n}(2),\ldots g_{2n}(2n-1),g_{2n}(2n)=0$$
respectively.

We define the surgeries as follows: Given a non crossing partition $\mathcal{A}\in NC_{n}$ such that $\mathcal{A}= \{m_{1},m_{2},\ldots m_{l}\}$ and $i_{l,1}< i_{l,2}<\ldots i_{l,k_{l}}$ are all the elements in the class $m_{l}$, we define:

$$\vec{\mathcal{A}}= \bigsqcup_{l} \overrightarrow{[g_{2n}(2i_{l,1}),g_{2n}(2i_{l,2}-1)]}\sqcup \overrightarrow{[g_{2n}(2i_{l,2}),g_{2n}(2i_{l,3}-1)]}\sqcup\ldots \overrightarrow{[g_{2n}(2i_{l,k_{l}}),g_{2n}(2i_{l,1}-1)]}$$
where $\overrightarrow{[b,a]}$ denotes the oriented line segment from the point $b$ to the point $a$ such that $a,b\in\C$. See that $\vec{\mathcal{A}}$ is a disjoint union of oriented segments starting at some $b_{i}$ point to another $a_{j}$ point in the circle. We define the oriented link diagram surgeries:
\begin{eqnarray}\label{surgeriesDef}
L_{1}^{\mathcal{A}} &=& (L\cap R1)\sqcup \varphi_{2}^{-1}(\vec{\mathcal{A}}) \\
L_{2}^{\mathcal{A}} &=& (L\cap R2)\sqcup \varphi_{1}^{-1}(\vec{\mathcal{A}}^{op})
\end{eqnarray}
where $\vec{\mathcal{A}}^{op}$ is the diagram $\vec{\mathcal{A}}$ with the opposite orientation.

Because the set of non crossing partitions is closed under the Dihedral group action, the set of these surgeries is independent of the orientation and marking of the Jordan curve $C$ and modulo ambient isotopy it is independent of the regions $R1,R2$ and the diffeomorphisms $\varphi_{1},\varphi_{2}$ previously chosen; i.e. Modulo ambient isotopy, the collection of surgeries only depends on the admissible cut as it was defined.

Because our construction neither modify the existing crossing nor adds new ones, we have that the number of positive/negative crossings $l_{1}^{\pm}$ is the same for all the surgeries $\{L_{1}^{\mathcal{A}}\}$ and a similar result for $\{L_{2}^{\mathcal{A}}\}$. Moreover,

\begin{equation}\label{relationSigns}
l^{\pm}= l_{1}^{\pm}+ l_{2}^{\pm}
\end{equation}
where $l^{\pm}$ denote the number of positive/negative crossings of the oriented link $L$. In other words, the number of crossings and the writhe are additive respect to the cut.

The surgeries $L_{i}^{\mathcal{A}}$ have the common set of crossings $\chi_{i}:= \chi\cap R_{i}$ and the set of crossings $\chi$ of the oriented link $L$ is the disjoint union:
$$\chi=\chi_{1}\sqcup\chi_{2}$$
This way, the set of states decomposes as:
$$\{0,1\}^{\chi}= \{0,1\}^{\chi_{1}}\times \{0,1\}^{\chi_{2}}$$
In other words, every state $\alpha\in \{0,1\}^{\chi}$ gives a pair of states $\alpha_{i}\in \{0,1\}^{\chi_{i}}$ for $i=1,2$ such that $\alpha= (\alpha_{1},\alpha_{2})$.

From now on, we will denote simply by $L_{1}$ and $L_{2}$ the respective $full$-surgeries $L_{1}^{full}$ and $L_{2}^{full}$ where $full$ denotes the partition $\{\{1\},\{2\},\ldots \{n\}\}$. For $i=1,2$, we define the following maps:
$$\mathcal{C}_{i}:\{0,1\}^{\chi_{i}}\rightarrow NC_{n}$$
such that:
\begin{itemize}
\item If the set of crossings $\chi_{i}$ is empty, define the equivalence relation $\sim_{\emptyset}$ such that $i\sim_{\emptyset} j$ if $a_{i}$ and $a_{j}$ belong to the same connected component of the $full$-surgery $L_{i}$. We define:
$$\mathcal{C}_{i}(*)= \{1,2,\ldots n\}/\sim_{\emptyset}$$
where $*$ is the only state of $\{0,1\}^{\emptyset}$.

\item If the set of crossings $\chi_{i}$ is nonempty, for every state $\beta\in \{0,1\}^{\chi_{i}}$ define the equivalence relation $\sim_{\beta}$ such that $i\sim_{\beta} j$ if $a_{i}$ and $a_{j}$ belong to the same connected component of the smoothing $\mathcal{S}_{\beta}(L_{i})$. We define:
$$\mathcal{C}_{i}(\beta)= \{1,2,\ldots n\}/\sim_{\beta}$$
\end{itemize}
Because different connected components do not intersect, the obtained partition must be non crossing.

Given an admissible cut of a link diagram $L$, we say that a circle in the smoothing $\mathcal{S}_{\alpha}(L)$ is \textit{inner} if it doesn't contain any point of the cut; i.e. it doesn't contain any of the points $a_{i},\ b_{j}$. Otherwise, the circle will be called \textit{outer}. The set of these circles will be denoted by $Inner_{\alpha}(L)$ and $Outer_{\alpha}(L)$ respectively and their disjoint union is the smoothing $\mathcal{S}_{\alpha}(L)$.

\begin{prop}\label{Bijection_Natural}
Given a link diagram $L$ and an admissible cut, for every state $\alpha=(\alpha_{1},\alpha_{2})$ we have natural one to one maps:
$$\mathcal{S}_{\alpha}(L)\cong Inner_{\alpha_{1}}(L_{1})\sqcup \mathcal{S}_{\alpha_{2}}\left( L_{2}^{\mathcal{C}_{1}(\alpha_{1})}\right)
\cong \mathcal{S}_{\alpha_{1}}\left( L_{1}^{\mathcal{C}_{2}(\alpha_{2})}\right)\sqcup Inner_{\alpha_{2}}(L_{2})$$
\end{prop}
\begin{proof}
We prove one bijection only for the second is verbatim. For every crossing of the link diagram $\chi$, take the smoothing neighborhood in Fig. \ref{Smoothing} small enough such that it belongs to one of the regions $R_{1}$ or $R_{2}$. Then, the set of circles in the smoothing $\mathcal{S}_{\alpha}(L)$ that belong to $R_{1}$ equals the set $Inner_{\alpha_{1}}(L_{1})$ hence the identity is a natural map between these sets. If a circle $S$ in the smoothing $\mathcal{S}_{\alpha}(L)$ doesn't belong to $R_{1}$ then its intersection with $R_{2}$ is nonempty. By construction of the map $\mathcal{C}_{1}$, there is a unique circle $S'$ in the smoothing $\mathcal{S}_{\alpha_{2}}\left( L_{2}^{\mathcal{C}_{1}(\alpha_{1})}\right)$ such that $S\cap R_{2}= S'\cap R_{2}$. Conversely, by the same argument, for every circle $S'$ in the smoothing $\mathcal{S}_{\alpha_{2}}\left( L_{2}^{\mathcal{C}_{1}(\alpha_{1})}\right)$ there is a unique circle $S$ in the smoothing $\mathcal{S}_{\alpha}(L)$ such that $S\cap R_{2}= S'\cap R_{2}$ and clearly doesn't belong to $R_{1}$. This correspondence defines a natural one to one map.
\end{proof}

By definition and a similar argument to the one in the proof of Proposition \ref{Bijection_Natural}, given a state $\alpha_{i}\in \{0,1\}^{\chi_{i}}$ we conclude that the sets $Inner_{\alpha_{i}}(L_{i}^{\mathcal{A}})$ are all equal:
$$Inner_{\alpha_{i}}(L_{i}^{\mathcal{A}})=Inner_{\alpha_{i}}(L_{i}^{\mathcal{B}})\ \ \ \textsl{for\ every}\ \mathcal{A},\mathcal{B}\in NC_{n}$$
In particular, for every non crossing partition $\mathcal{A}$ there is a natural bijection:
\begin{equation}\label{Bijection_Natural_II}
Inner_{\alpha}(L_{i})\cong Inner_{\alpha}(L_{i}^{\mathcal{A}})
\end{equation}
Because the bijections in Proposition \ref{Bijection_Natural} are natural, we have:

\begin{cor}\label{Bijection_Natural_IV}
Given a link diagram $L$ and an admissible cut, for every state $\alpha=(\alpha_{1},\alpha_{2})$ we have natural one to one maps:
$$Outer_{\alpha}(L)\cong Outer_{\alpha_{2}}\left( L_{2}^{\mathcal{C}_{1}(\alpha_{1})}\right)
\cong Outer_{\alpha_{1}}\left( L_{1}^{\mathcal{C}_{2}(\alpha_{2})}\right)$$
\end{cor}

\subsection{Order convention}\label{Order_Compatible}

Given an admissible cut of a link diagram $L$ and a state $\alpha$ we define the following total order relation on $Outer_{\alpha}\left(L\right)$:
$$S<S'\ \ \ \textsl{if}\ \ \ \min\{i\ \ \textsl{such\ that}\ \ a_{i}\in S\}<\min\{i\ \ \textsl{such\ that}\ \ a_{i}\in S'\}$$
for different circles $S$ and $S'$ in $Outer_{\alpha}\left(L\right)$.

Because of the natural bijection \eqref{Bijection_Natural_II}, we have the natural bijection:
\begin{equation}\label{Bijection_Natural_III}
\mathcal{S}_{\alpha_{i}}\left( L_{i}^{\mathcal{A}}\right)= Inner_{\alpha_{i}}\left( L_{i}^{\mathcal{A}}\right)\sqcup Outer_{\alpha_{i}}\left( L_{i}^{\mathcal{A}}\right)\cong Inner_{\alpha_{i}}\left( L_{i}\right)\sqcup Outer_{\alpha_{i}}\left( L_{i}^{\mathcal{A}}\right)
\end{equation}

For $i=1,2$ and for every state $\alpha_{i}\in \{0,1\}^{\chi_{i}}$ consider a total ordering on the set $Inner_{\alpha_{i}}(L_{i})$. We extend these orderings to total orderings on the smoothings as follows\footnote{We say that $A<B$ if $a<b$ for every pair of elements $a$ and $b$ in $A$ and $B$ respectively.}:
\begin{itemize}
\item For every surgery $L_{1}^{\mathcal{A}}$ and every state $\alpha_{1}\in \{0,1\}^{\chi_{1}}$ we extend this order to $\mathcal{S}_{\alpha_{1}}\left( L_{1}^{\mathcal{A}}\right)$ via \eqref{Bijection_Natural_III} such that $Inner_{\alpha_{1}}(L_{1}) < Outer_{\alpha_{1}}\left( L_{1}^{\mathcal{A}}\right)$.
\item For every surgery $L_{2}^{\mathcal{A}}$ and every state $\alpha_{2}\in \{0,1\}^{\chi_{2}}$ we extend this order to $\mathcal{S}_{\alpha_{2}}\left( L_{2}^{\mathcal{A}}\right)$ via \eqref{Bijection_Natural_III} such that $Inner_{\alpha_{2}}(L_{2}) > Outer_{\alpha_{2}}\left( L_{2}^{\mathcal{A}}\right)$.
\end{itemize}
This ordering will be called a \textit{compatible order respect to the cut}.

Finally, via Proposition \ref{Bijection_Natural} this convention defines a total order on $\mathcal{S}_{\alpha}(L)$. In effect, for every state $\alpha=(\alpha_{1}, \alpha_{2})$ there is a natural bijection:
$$\mathcal{S}_{\alpha}(L)\cong Inner_{\alpha_{1}}(L_{1})\sqcup Outer_{\alpha}(L) \sqcup Inner_{\alpha_{2}}(L_{2})$$
such that:
$$Inner_{\alpha_{1}}(L_{1}) < Outer_{\alpha}(L) < Inner_{\alpha_{2}}(L_{2})$$
This is well defined because of Corollary \ref{Bijection_Natural_IV}. This ordering will be called \textit{the order inherited from the compatible order respect to the cut}.

\section{Double Khovanov complex}

Consider partitions $\mathcal{A}$ and $\mathcal{B}$ in $\Gamma_{n}$. We write $\mathcal{A}\prec\mathcal{B}$ if $\mathcal{A}$ is finer than $\mathcal{B}$; i.e. if for every $a\in\mathcal{A}$ there is $b\in\mathcal{B}$ such that $a\subset b$. We define the products:

$$\mathcal{A}\vee \mathcal{B}= min\{\mathcal{C}\in \Gamma_{n}\ such\ that\ \mathcal{A}\prec \mathcal{C}\ and\ \mathcal{B}\prec \mathcal{C}\}$$
$$\mathcal{A}\wedge \mathcal{B}= max\{\mathcal{C}\in \Gamma_{n}\ such\ that\ \mathcal{A}\succ \mathcal{C}\ and\ \mathcal{B}\succ \mathcal{C}\}$$

See that $\Gamma_{n}^{\wedge}:=(\Gamma_{n},\wedge, trivial)$ and $\Gamma_{n}^{\vee}:=(\Gamma_{n}, \vee, full)$ are commutative monoids where $trivial=\{\{1,2,\ldots n\}\}$ and $full= \{\{1\},\{2\},\ldots \{n\}\}$ are the respective units. These products satisfy the following compatibility relations:
\begin{eqnarray*}
\mathcal{A}\wedge(\mathcal{B}\vee\mathcal{C})&=& (\mathcal{A}\wedge\mathcal{B})\vee(\mathcal{A}\wedge\mathcal{C}) \\
\mathcal{A}\vee(\mathcal{B}\wedge\mathcal{C})&=& (\mathcal{A}\vee\mathcal{B})\wedge(\mathcal{A}\vee\mathcal{C}) \\
\mathcal{A}\wedge full &=& full \\
\mathcal{A}\vee trivial &=& trivial
\end{eqnarray*}
for every tender of partitions $\mathcal{A}, \mathcal{B}$ and $\mathcal{C}$. The subset of non crossing partitions $NC_{n}$ is closed under $\wedge$ but not under $\vee$; i.e. it is a submonoid of $\Gamma_{n}^{\wedge}$ but not of $\Gamma_{n}^{\vee}$:
$$(NC_{n}, \wedge, trivial)<\Gamma_{n}^{\wedge}$$

Now, consider an admissible cut $C$ of an oriented link diagram $L$. Just as in the previous section, we have the intersection points:
$$a_{1},b_{1},a_{2},b_{2},\ldots a_{n},b_{n}$$

\begin{defi}\label{Bigraded_vector_space}
For every non crossing partition $\mathcal{A}\in NC_{n}$ and $i=1,2$, we define the bigraded vector space:
$$[\![L_{i}]\!]_{\mathcal{A}}:= \bigoplus_{\substack{\alpha\in\{0,1\}^{\chi_{i}} \\ \mathcal{C}_{i}(\alpha)= \mathcal{A}}} V^{\otimes(k_{i}(\alpha)-|\mathcal{A}|)}\{r(\alpha)\}[r(\alpha)]$$
\end{defi}

See that, because the set of circles containing the intersection points is a subset of the set of circles of the smoothing $\mathcal{S}_{\alpha}(L_{i})$ for every state $\alpha$, we have that $|\mathcal{C}_{i}(\alpha)|\leq k_{i}(\alpha)$ for every state $\alpha$. In particular, the expression above is well defined. The Khovanov version of the bigraded vector space defined in \ref{Bigraded_vector_space} is obtained by shifting degrees:
$$\mathcal{C}(L_{i})_{\mathcal{A}}:=[\![L_{i}]\!]_{\mathcal{A}}[-l_{i}^{-}]\{l_{i}^{+}-2l_{i}^{-}\}$$
See also that if $L_{i}$ has no crossings, then \footnote{This can be seen as the categorification of the delta Kronecker.}:

\begin{equation}\label{remarkdelta}
\mathcal{C}(L_{i})_{\mathcal{A}}:=
\left\{
\begin{array}{c l}
 k		&\ \ \ \ \ \ if\ \  \mathcal{A}=\mathcal{C}_{i}(*) \\
 (0)	&\ \ \ \ \ \ if\ \  \mathcal{A}\neq\mathcal{C}_{i}(*)
\end{array}
\right.
\end{equation}

\begin{lema}\label{Decomposition1}
We have the following isomorphism of bigraded vector spaces:
$$\mathcal{C}(L)\cong \bigoplus_{\mathcal{A},\mathcal{B}\in NC_{n}} \mathcal{C}(L_{1})_{\mathcal{A}}\otimes a^{\mathcal{A}\mathcal{B}}\otimes \mathcal{C}(L_{2})_{\mathcal{B}}$$
such that:
$$a^{\mathcal{A}\mathcal{B}}:= V^{\otimes(n+|\mathcal{A}\vee\mathcal{B}|-|\mathcal{A}\wedge\mathcal{B}|)}$$
\end{lema}
\begin{proof}
Suppose that $\alpha=(\alpha_{1},\alpha_{2})$ is a state such that $\mathcal{C}_{1}(\alpha_{1})= \mathcal{A}$ and $\mathcal{C}_{2}(\alpha_{2})= \mathcal{B}$. Then, the number of circles in the smoothing $\mathcal{S}_{\alpha_{1}}(L_{1})$ that contain the intersection points is $|\mathcal{A}|$ and similarly, the number of circles in the smoothing $\mathcal{S}_{\alpha_{2}}(L_{2})$ that contain the intersection points is $|\mathcal{B}|$.

Consider a tubular neighborhood $C_{\varepsilon}$ of the admissible cut $C$ such that its boundary consists of two admissible cuts:
$$\partial C_{\varepsilon}= C_{1}\sqcup C_{2}$$
Consider $\varepsilon$ small enough such that:
\begin{itemize}
\item There is $\pi:C_{\varepsilon}\rightarrow C$ a vector bundle of $C$.
\item $\pi$ has a unique continuous extension to the boundary $\partial C_{\varepsilon}$.
\item $C_{\varepsilon}\cap L$ is a, necessarily finite, union of fibers of the vector bundle.
\end{itemize}
In particular, the intersection points:
$$a_{1}^{i},b_{1}^{i},a_{2}^{i},b_{2}^{i},\ldots a_{n}^{i},b_{n}^{i}$$
of $C_{i}$ and $L$ project via the extension of $\pi$ to the intersection points of $C$ and $L$. By definition of the $full$-surgery, modulo ambient isotopy, we get $\mathcal{S}_{\alpha_{1}}(L_{1})\sqcup \mathcal{S}_{\alpha_{2}}(L_{2})$ by the following surgery on $\mathcal{S}_{\alpha}(L)$: Substitute every $a_{j}^{1},b_{j}^{1},a_{j}^{2},b_{j}^{2}$ smoothing of $C_{\varepsilon}\cap L$ by the opposite Kauffman smoothing. Because $C_{\varepsilon}\cap L$ has no crossings, taking the smoothing neighborhoods and the parameter $\varepsilon$ small enough we have that $C_{\varepsilon}\cap L= C_{\varepsilon}\cap \mathcal{S}_{\alpha}(L)$.

In particular, the number of circles in the smoothing $\mathcal{S}_{\alpha}(L)$ containing the intersection points $a_{1},b_{1},\ldots a_{n},b_{n}$ is
$n+|\mathcal{A}\vee\mathcal{B}|-|\mathcal{A}\wedge\mathcal{B}|$. Because the number of circles in the respective smoothings not containing the intersection points is additive respect to the cut, we have the following formula:

$$k(\alpha)= \underbrace{k(\alpha_{1})-|\mathcal{A}|}_\text{1.} + \underbrace{k(\alpha_{2})- |\mathcal{B}|}_\text{2.}+ \underbrace{n+|\mathcal{A}\vee\mathcal{B}|-|\mathcal{A}\wedge\mathcal{B}|}_\text{3.}$$
such that:
\begin{enumerate}
\item = Number of circles in the smoothing $\mathcal{S}_{\alpha_{1}}(L_{1})$ not containing the intersection points. \\
\item = Number of circles in the smoothing $\mathcal{S}_{\alpha_{2}}(L_{2})$ not containing the intersection points. \\
\item = Number of circles in the smoothing $\mathcal{S}_{\alpha}(L)$ containing the intersection points.
\end{enumerate}

\noindent Because $r(\alpha)= r(\alpha_{1})+r(\alpha_{2})$ we have:

\begin{eqnarray*}
[\![L]\!] &\cong& \bigoplus_{\alpha\in\{0,1\}^{\chi} } V^{\otimes k(\alpha)}\{r(\alpha)\} \\
&\cong& \bigoplus_{\mathcal{A},\mathcal{B}\in NC_{n}} \bigoplus_{\substack{\alpha_{1}\in\{0,1\}^{\chi_{1}} \\ \mathcal{C}_{1}(\alpha_{1})= \mathcal{A}}}\bigoplus_{\substack{\alpha_{2}\in\{0,1\}^{\chi_{2}} \\ \mathcal{C}_{2}(\alpha_{2})= \mathcal{B}}} V^{\otimes(k(\alpha_{1})-|\mathcal{A}| + k(\alpha_{2})- |\mathcal{B}|+ n+|\mathcal{A}\vee\mathcal{B}|-|\mathcal{A}\wedge\mathcal{B}|)}\{r(\alpha_{1})+r(\alpha_{2})\} \\
&\cong& \bigoplus_{\mathcal{A},\mathcal{B}\in NC_{n}} [\![L_{1}]\!]_{\mathcal{A}}\otimes V^{\otimes(n+|\mathcal{A}\vee\mathcal{B}|-|\mathcal{A}\wedge\mathcal{B}|)}\otimes [\![L_{2}]\!]_{\mathcal{B}}
\end{eqnarray*}
Because of relation \ref{relationSigns}, after shifting degrees we have the result.
\end{proof}

\begin{cor}\label{Decomposition2}
We have the following isomorphism of bigraded vector spaces:
$$\mathcal{C}\left( L_{i}^{\mathcal{A}}\right)\cong \bigoplus_{\mathcal{B}\in NC_{n}} a^{\mathcal{A}\mathcal{B}}\otimes \mathcal{C}\left(L_{i}\right)_{\mathcal{B}}$$
such that:
$$a^{\mathcal{A}\mathcal{B}}:= V^{\otimes(n+|\mathcal{A}\vee\mathcal{B}|-|\mathcal{A}\wedge\mathcal{B}|)}$$
\end{cor}
\begin{proof}
Recall the definition of the surgeries \eqref{surgeriesDef}. The curve $C$ is an admissible cut of the surgery $L_{i}^{\mathcal{A}}$ and by the previous Lemma we have:

$$\mathcal{C}\left( L_{i}^{\mathcal{A}}\right) \cong \bigoplus_{\mathcal{C},\mathcal{B}\in NC_{n}} \mathcal{C}\left( full_{i}^{\mathcal{A}} \right)_{\mathcal{C}}\otimes V^{\otimes(n+|\mathcal{C}\vee\mathcal{B}|-|\mathcal{C}\wedge\mathcal{B}|)}\otimes \mathcal{C}\left( L_{i} \right)_{\mathcal{B}}$$
such that:
$$full_{1}^{\mathcal{A}}= \varphi_{1}^{-1}(\overrightarrow{full}^{op})\sqcup\varphi_{2}^{-1}(\vec{\mathcal{A}})$$
$$full_{2}^{\mathcal{A}}= \varphi_{2}^{-1}(\overrightarrow{full})\sqcup\varphi_{1}^{-1}(\vec{\mathcal{A}}^{op})$$
Because the oriented links $full_{i}^{\mathcal{A}}$ have no crossings, by equation \eqref{remarkdelta} we have:
$$\mathcal{C}\left( full_{i}^{\mathcal{A}}\right)_{\mathcal{C}}=
\left\{
\begin{array}{c l}
 k		&\ \ \ \ \ \ if\ \  \mathcal{A}=\mathcal{C} \\
 (0)	&\ \ \ \ \ \ if\ \  \mathcal{A}\neq\mathcal{C}
\end{array}
\right.$$
and the result follows.
\end{proof}

In the Grothendieck ring of graded vector spaces, the determinant of the matrix $A:=\left(a^{\mathcal{A}\mathcal{B}}\right)$ in Lemma \ref{Decomposition1} and Corollary \ref{Decomposition2} is non singular for its diagonal is the higher degree term. As an immediate Corollary we have:

\begin{cor}\label{}
We have the following identity in the Grothendieck ring of the abelian category of bigraded vector spaces\ \footnote{We abuse of notation neither writing the forgetful functor on the respective complexes nor the equivalence classes.}:
\begin{eqnarray}
\mathcal{C}(L)\ \det(A) &=& \sum_{\mathcal{A},\mathcal{B}\in NC_{n}} \mathcal{C}\left( L_{1}^{\mathcal{A}}\right)\ A_{\mathcal{A}\mathcal{B}}\ \mathcal{C}\left( L_{2}^{\mathcal{B}}\right) \label{First_Groth_eq} \\
\mathcal{C}(L_{i})_{\mathcal{A}}\ \det(A) &=& \sum_{\mathcal{A},\mathcal{B}\in NC_{n}}\ A_{\mathcal{A}\mathcal{B}}\ \mathcal{C}\left( L_{i}^{\mathcal{B}}\right) \label{Second_Groth_eq}
\end{eqnarray}
where $A_{\mathcal{A}\mathcal{B}}$ is the $\mathcal{A}\mathcal{B}$-cofactor of the matrix $A:=(a^{\mathcal{A}\mathcal{B}})$ such that:
$$a^{\mathcal{A}\mathcal{B}}:= V^{\otimes(n+|\mathcal{A}\vee\mathcal{B}|-|\mathcal{A}\wedge\mathcal{B}|)}$$
\end{cor}

In particular, by formula \eqref{Second_Groth_eq} the bigraded vector spaces $\mathcal{C}(L_{i})_{\mathcal{A}}$ can be calculated entirely from the Khovanov complexes of the surgeries $L_{i}^{\mathcal{A}}$.

\begin{defi}\label{Double_Khovanov_complex}
Given an admissible cut $C$ of an oriented link diagram $L$, consider the Khovanov complexes of the surgeries:
$$\mathcal{C}\left( L_{i}^{\mathcal{A}}\right)=\left(\mathcal{C}\left( L_{i}^{\mathcal{A}}\right)^{r},\ d_{i}^{\mathcal{A},r}\right)$$
such that a compatible order respect to the cut was taken on the smoothings of the surgeries (See section \ref{Order_Compatible}).
We define the Khovanov double complex $\mathcal{CC}(L;C)$ as follows:
\begin{equation}\label{Double_Complex_spaces}
\mathcal{CC}(L;C)^{s,t}:= \bigoplus_{\mathcal{A}\in NC_{n}} \mathcal{C}(L_{1})_{\mathcal{A}}^{s}\otimes \mathcal{C}\left( L_{2}^{\mathcal{A}} \right)^{t}
\cong \bigoplus_{\mathcal{A}\in NC_{n}} \mathcal{C}\left( L_{1}^{\mathcal{A}}\right)^{s}\otimes \mathcal{C}(L_{2})_{\mathcal{A}}^{t}
\end{equation}
and the vertical and horizontals differentials are given by:
\begin{eqnarray}\label{Double_Complex_morphisms}
d^{s,t}_{vert} &:=& \bigoplus_{\mathcal{A}\in NC_{n}} id_{1,\mathcal{A}}^{\ s}\otimes d_{2}^{\mathcal{A},t} \\
d^{s,t}_{hor} &:=& \bigoplus_{\mathcal{A}\in NC_{n}} d_{1}^{\mathcal{A},s}\otimes id_{2,\mathcal{A}}^{\ t}
\end{eqnarray}
\end{defi}

The isomorphism in \ref{Double_Complex_spaces} is an immediate consequence of Lemma \ref{Decomposition1} and Corollary \ref{Decomposition2}. Is clear from the definition that $d_{vert}$ and $d_{hor}$ are differentials. To prove that these differentials commute and $\mathcal{CC}(L;C)$ is in fact a double complex; i.e. $d_{vert}\circ d_{hor}= d_{vert}\circ d_{vert}$, it is enough to prove that the total complex of $\mathcal{CC}(L;C)$ is actually a chain complex \footnote{This would be false if a compatible order respect to the cut were not taken in definition \ref{Double_Khovanov_complex}.}. This is the content of the next Proposition:

\begin{prop}\label{Splitting_Khovanov_complex}
Given an admissible cut $C$ of an oriented link diagram $L$, the total complex\ \footnote{Consider a double complex $A= (A^{p,q},\ d^{p,q}_{hor},\ d^{p,q}_{vert})$. The total complex $Tot(A)$ is the complex such that $Tot(A)^{k}= \bigoplus_{p+q=k}A^{p,q}$ and: $$d^{k}= \bigoplus_{p+q=k} \left( d^{p,q}_{hor} +(-1)^{p}  d^{p,q}_{vert}\right)$$} of the Khovanov double complex is isomorphic to the Khovanov complex:
$$\mathcal{C}(L)\cong Tot\left(\mathcal{CC}(L;C)\right)$$
such that the Khovanov complex is constructed with the order inherited from the compatible order respect to the cut taken in the construction of the double complex (See section \ref{Order_Compatible} and definition \ref{Double_Khovanov_complex}).
\end{prop}
\begin{proof}
We will follow the notation and definitions in the preliminary section \ref{Khovanov_Nutshell}. Consider a morphism of states $\alpha\in \{0,1,*\}^{\chi}$ such that $*$ only appears once:
$$|\{i\ such\ that\ \alpha(i)=*\}|=1$$
Consider the canonical decomposition $\alpha= (\alpha_{1},\alpha_{2})$. Because $*$ only appears once, it must be in $\alpha_{1}$ or in $\alpha_{2}$ but not both. We will say that $\alpha$ is a \textit{left} morphism if $*$ is in $\alpha_{1}$ and a \textit{right} morphism otherwise.

Suppose without loss of generality that $\alpha$ is a right morphism; i.e. $*$ is in $\alpha_{2}$. Consider the smoothings $\mathcal{S}_{s(\alpha)}(L)$ and $\mathcal{S}_{t(\alpha)}(L)$. Because $s(\alpha)=(\alpha_{1},s(\alpha_{2}))$ and $t(\alpha)=(\alpha_{1},t(\alpha_{2}))$, the circles not containing the intersection points in the smoothing $\mathcal{S}_{\alpha_{1}}(L_{1})$ remain the same; i.e. they are common to both smoothings $\mathcal{S}_{s(\alpha)}(L)$ and $\mathcal{S}_{t(\alpha)}(L)$. The set of these circles will be denoted $Inner_{\alpha_{1}}(L_{1})$. The set of circles containing the intersection points in the smoothing $\mathcal{S}_{\alpha_{1}}(L_{1})$ will be denoted by $Outer_{\alpha_{1}}(L_{1})$. See that $|Outer_{\alpha_{1}}(L_{1})|= |\mathcal{C}_{1}(\alpha_{1})|$ hence:

$$|Inner_{\alpha_{1}}(L_{1})|= k(\alpha_{1})-|\mathcal{C}_{1}(\alpha_{1})|$$

By Proposition \ref{Bijection_Natural}, there are natural one to one correspondences:
$$\mathcal{S}_{s(\alpha)}(L)\xrightarrow{\ \ \ 1:1\ \ \ } Inner_{\alpha_{1}}(L_{1}) \sqcup \mathcal{S}_{s(\alpha_{2})}(L_{2}^{\mathcal{C}_{1}(\alpha_{1})})$$
$$\mathcal{S}_{t(\alpha)}(L)\xrightarrow{\ \ \ 1:1\ \ \ } Inner_{\alpha_{1}}(L_{1}) \sqcup \mathcal{S}_{t(\alpha_{2})}(L_{2}^{\mathcal{C}_{1}(\alpha_{1})})$$
In particular, because we have taken compatible orderings in the smoothings, we have the following identity:
$$d_{\alpha}= id^{\otimes(k(\alpha_{1})-|\mathcal{C}_{1}(\alpha_{1})|)}\otimes d_{\alpha_{2}}^{\mathcal{C}_{1}(\alpha_{1})}$$
where $d_{\alpha}$ and $d_{\alpha_{2}}^{\mathcal{C}_{1}(\alpha_{1})}$ are the categorifications of the morphisms $\alpha$ and $\alpha_{2}$ in the category of smoothings of $L$ and $L_{2}^{\mathcal{C}_{1}(\alpha_{1})}$ respectively; i.e. $d_{\alpha}= F(\alpha)$ and $d_{\alpha_{2}}^{\mathcal{C}_{1}(\alpha_{1})}= F(\alpha_{2})$ where $F$ is the topological quantum field theory functor defined in section \ref{Khovanov_Nutshell}.

Because of the following relation:
$$sg(\alpha)= \sum_{i<j(*,\alpha)}\alpha_{i}= r(\alpha_{1})+ \sum_{i<j(*,\alpha_{2})}\alpha_{2,i}= r(\alpha_{1})+ sg(\alpha_{2})$$
we have:
$$(-1)^{sg(\alpha)}\ d_{\alpha}= (-1)^{r(\alpha_{1})}\ id^{\otimes (k(\alpha_{1})-|\mathcal{C}_{1}(\alpha_{1})|)}\otimes (-1)^{sg(\alpha_{2})}\ d_{\alpha_{2}}^{\mathcal{C}_{1}(\alpha_{1})}$$

Summing over the set of all right morphisms $\alpha$, we get the following expression:

\begin{eqnarray*}
d^{right,\ k} &=& \bigoplus_{\substack{ \alpha \textsl{is\ a\ right\ morphism}\\ r(\alpha(0))=k}} (-1)^{sg(\alpha)}\ d_{\alpha} =
\bigoplus_{k=i+j}\ \ \ \bigoplus_{\substack{\alpha_{1} \\ r(\alpha_{1})=i }}\ \ \ \bigoplus_{\substack{\alpha_{2} \\ r(\alpha_{2}(0))=j }} (-1)^{sg(\alpha)}\ d_{\alpha} \\
&=&\bigoplus_{k=i+j}\bigoplus_{\substack{\alpha_{1} \\ r(\alpha_{1})=i }}\ \left((-1)^{i}\ id^{\otimes (k(\alpha_{1})-|\mathcal{C}_{1}(\alpha_{1})|)}\otimes \left(\bigoplus_{\substack{\alpha_{2} \\ r(\alpha_{2}(0))=j }}\ (-1)^{sg(\alpha_{2})}\ d_{\alpha_{2}}^{\mathcal{C}_{1}(\alpha_{1})}\right)\right) \\
&=&\bigoplus_{\mathcal{A}}\bigoplus_{k=i+j}\bigoplus_{\substack{ \alpha_{1} \\ r(\alpha_{1})=i \\ \mathcal{C}_{1}(\alpha_{1})=\mathcal{A}}}\ \left((-1)^{i}\ id^{\otimes (k(\alpha_{1})-|\mathcal{A}|)}\otimes \left(\bigoplus_{{\substack{\alpha_{2} \\ r(\alpha_{2}(0))=j }}}\ (-1)^{sg(\alpha_{2})}\ d_{\alpha_{2}}^{\mathcal{A}}\right)\right) \\
&=&\bigoplus_{\mathcal{A}}\bigoplus_{k=i+j}\ (-1)^{i}\left(\bigoplus_{\substack{ \alpha_{1} \\ r(\alpha_{1})=i \\ \mathcal{C}_{1}(\alpha_{1})=\mathcal{A}}}\ id^{\otimes (k(\alpha_{1})-|\mathcal{A}|)}\right)\otimes \left(\bigoplus_{{\substack{\alpha_{2} \\ r(\alpha_{2}(0))=j }}}\ (-1)^{sg(\alpha_{2})}\ d_{\alpha_{2}}^{\mathcal{A}}\right) \\
&=&\bigoplus_{\mathcal{A}}\bigoplus_{k=i+j}\ (-1)^{i}\ id_{1,\mathcal{A}}^{\ i}\otimes d_{2}^{\mathcal{A},j}
\end{eqnarray*}

We have proved the following isomorphism of chain complexes \footnote{Recall that given chain complexes $(A^{\bullet},f^{\bullet})$ and $(B^{\bullet},g^{\bullet})$, the $k$-th morphism of the chain complex $A\otimes B$ is: $$h^{k}=\bigoplus_{i+j=k} \left(f^{i}\otimes id + (-1)^{i} id\otimes g^{j}\right)$$} \footnote{In the formula, we think of the bigraded vector space $[\![L_{1}]\!]_{\mathcal{A}}$ as a chain complex of graded vector spaces with zero morphisms.}:
$$[\![L]\!]^{right}\cong \bigoplus_{\mathcal{A}\in NC_{n}} [\![L_{1}]\!]_{\mathcal{A}}\otimes [\![L_{2}^{\mathcal{A}}]\!]$$
where $[\![L]\!]^{right}$ denotes the chain complex $[\![L]\!]$ without the categorification of the left morphisms.
The formula for $[\![L]\!]^{left}$ follows verbatim and after shifting degrees we have the result.
\end{proof}

As in the Khovanov complex, different compatible orders respect to the cut give homotopic double chain complexes. By Theorem \ref{Khovanov_Homotopy}, we have the following Corollary:

\begin{cor}
Given an admissible cut $C$ of an oriented link diagram $L$, the total complex of the Khovanov double complex is homotopic to the Khovanov complex:
$$\mathcal{C}(L)\simeq Tot\left(\mathcal{CC}(L;C)\right)$$
\end{cor}

We summarize our results in the following Theorem:

\begin{teo}
Given an admissible cut $C$ of an oriented link diagram $L$, there is a double chain complex whose rows and columns are linear combinations of the Khovanov complexes of the surgeries $L_{1}^{\mathcal{A}}$ and $L_{2}^{\mathcal{B}}$ respectively in the Grothendieck ring of chain complexes such that its total complex is homotopic to the Khovanov complex of $L$.
\end{teo}

Notice that in the case where $n=0$, the double Khovanov complex is just the tensor product double complex and by Proposition \ref{Splitting_Khovanov_complex} we have the Khovanov complex factorization of the disjoint union (\cite{Khovanov_original}, section 7.4):
$$\mathcal{C}(L_{1}\sqcup L_{2})\cong Tot\left(\mathcal{CC}(L_{1}\sqcup L_{2};C)\right) \cong \mathcal{C}(L_{1})\otimes \mathcal{C}(L_{2})$$
In particular, from the K\"unneth formula in characteristic zero, we have the factorization of the Khovanov homology of the disjoint union:
$$Kh(L_{1}\sqcup L_{2})\cong Kh(L_{1})\otimes Kh(L_{2})$$

\begin{remark}
In the $n=1$ case, resolving the double point, the double Khovanov complexes of the disjoint union, connected sum and double point diagrams can be arranged in a short exact sequence as the one in expression $(169)$, section 7.4 in Khovanov's original paper \cite{Khovanov_original}. Evaluating the abelian groups, the subsequent long homology exact sequence involves the Khovanov homology of the disjoint union and connected sum of the tangles only.
\end{remark}

For the general case, in characteristic zero we have the following spectral sequence:

\begin{teo}
Given an admissible cut $C$ of an oriented link diagram $L$, there is a spectral sequence of graded vector spaces converging to the Khovanov homology $Kh_{\Q}(L):= H(\mathcal{C}(L),\Q)$:
$$E^{2}_{p,q}\Rightarrow Kh_{\Q}(L)$$
such that:
\begin{equation}\label{spectral_sequence}
E^{2}_{p,q}\cong H^{p}\left(\bigoplus_{\mathcal{A}\in NC_{n}} \mathcal{C}(L_{1})_{\mathcal{A}}^{\bullet}\otimes Kh_{\Q}^{q}(L_{2}^{\mathcal{A}}),\ \bigoplus_{\mathcal{A}\in NC_{n}} \left[ d_{1}^{\mathcal{A},\bullet}\otimes id_{2,\mathcal{A}}^{\ q}\right]\right)
\end{equation}
\end{teo}
\begin{proof}
To simplify the notation, we will omit the symbol $\bigoplus_{\mathcal{A}\in NC_{n}}$ every time a super and a subindex repeat\footnote{This is the Einstein convention on repeated covariant and contravariant indexes.}. Consider the double complex of Proposition \ref{Splitting_Khovanov_complex}. Because this double complex is concentrated in the first quadrant, its spectral sequence converges to the cohomology of its total complex. Considering the horizontal filtration of the double complex, we have for the first page:
$$E^{1}_{p,q}:= H^{q}\left(\mathcal{C}(L_{1})_{\mathcal{A}}^{p}\otimes \mathcal{C}\left( L_{2}^{\mathcal{A}}\right)^{\bullet},\ id_{2,\mathcal{A}}^{\ p}\otimes d_{2}^{\mathcal{A},\bullet}\right)$$
Because we are taking cohomology in zero characteristic field $\Q$ coefficients, we actually have:
$$E^{1}_{p,q}\cong \mathcal{C}(L_{1})_{\mathcal{A}}^{p}\otimes H^{q}\left(\mathcal{C}\left( L_{2}^{\mathcal{A}}\right)^{\bullet},\ d_{2}^{\mathcal{A},\bullet}\right)= \mathcal{C}(L_{1})_{\mathcal{A}}^{p}\otimes H^{q}\left(\mathcal{C}\left( L_{2}^{\mathcal{A}}\right)\right)$$
We have the morphism of chain complexes \footnote{Watch out the super and sub indexes between the space and the morphism, they don't match.}:
$$\xymatrix{
    \mathcal{C}(L_{1})_{\mathcal{A}}^{p}\otimes \mathcal{C}\left( L_{2}^{\mathcal{A}}\right)^{\bullet}  \ar[rrr]^{d_{1}^{\mathcal{A},p}\otimes id_{2,\mathcal{A}}^{\bullet}} & & &	\mathcal{C}(L_{1})_{\mathcal{A}}^{p+1}\otimes \mathcal{C}\left( L_{2}^{\mathcal{A}}\right)^{\bullet} }$$
The cohomology of the morphism is just the restriction at the level of representatives and we will denote it as follows:
$$\xymatrix{
    \mathcal{C}(L_{1})_{\mathcal{A}}^{p}\otimes H^{q}\left(\mathcal{C}\left( L_{2}^{\mathcal{A}}\right)\right)  \ar[rrr]^{\left[ d_{1}^{\mathcal{A},p}\otimes id_{2,\mathcal{A}}^{\ q}\right]} & & &	\mathcal{C}(L_{1})_{\mathcal{A}}^{p+1}\otimes H^{q}\left(\mathcal{C}\left( L_{2}^{\mathcal{A}}\right)\right) }$$
Because $d_{1}^{\mathcal{A}}$ is a differential, the above complex is a chain complex respect to $p$. Taking the cohomology respect to $p$ we have the second page:
$$E^{2}_{p,q}\cong H^{p}\left(\mathcal{C}(L_{1})_{\mathcal{A}}^{\bullet}\otimes H^{q}\left(\mathcal{C}\left( L_{2}^{\mathcal{A}}\right)\right),\ \left[ d_{1}^{\mathcal{A},\bullet}\otimes id_{2,\mathcal{A}}^{\ q}\right]\right)$$
The proof is complete.
\end{proof}

\begin{cor}\label{Thevenin}
Given an admissible cut $C$ of an oriented link diagram $L$, if there is a tangle $T_{2}$ such that the Khovanov homology of its closures is isomorphic to the Khovanov homology of the surgeries $L_{2}^{\mathcal{A}}$ respectively; i.e. $Kh(T_{2}^{\mathcal{A}})\cong Kh(L_{2}^{\mathcal{A}})$, then:
$$Kh(L)\cong Kh\left( D_{1}( T_{2})\right)$$
where $D$ is the corresponding planar arc operation\footnote{See section 5 in \cite{Tangles_BN} for a definition of planar algebras and tangle compositions.} of the tangle $T_{1}$ obtained from the cut $C$ \footnote{I like to see this corollary as the analog of Thevenin's and Norton's Theorem in electrical circuit theory.}.
\end{cor}

We define a \textit{Half Solomon} and \textit{Hopf connected sum} link diagram as the pair given by the link diagram and the admissible cut shown in the left and right hand side of Fig. \ref{Half_Solomon} respectively.

\begin{figure}
\begin{center}
  \includegraphics[width=0.6\textwidth]{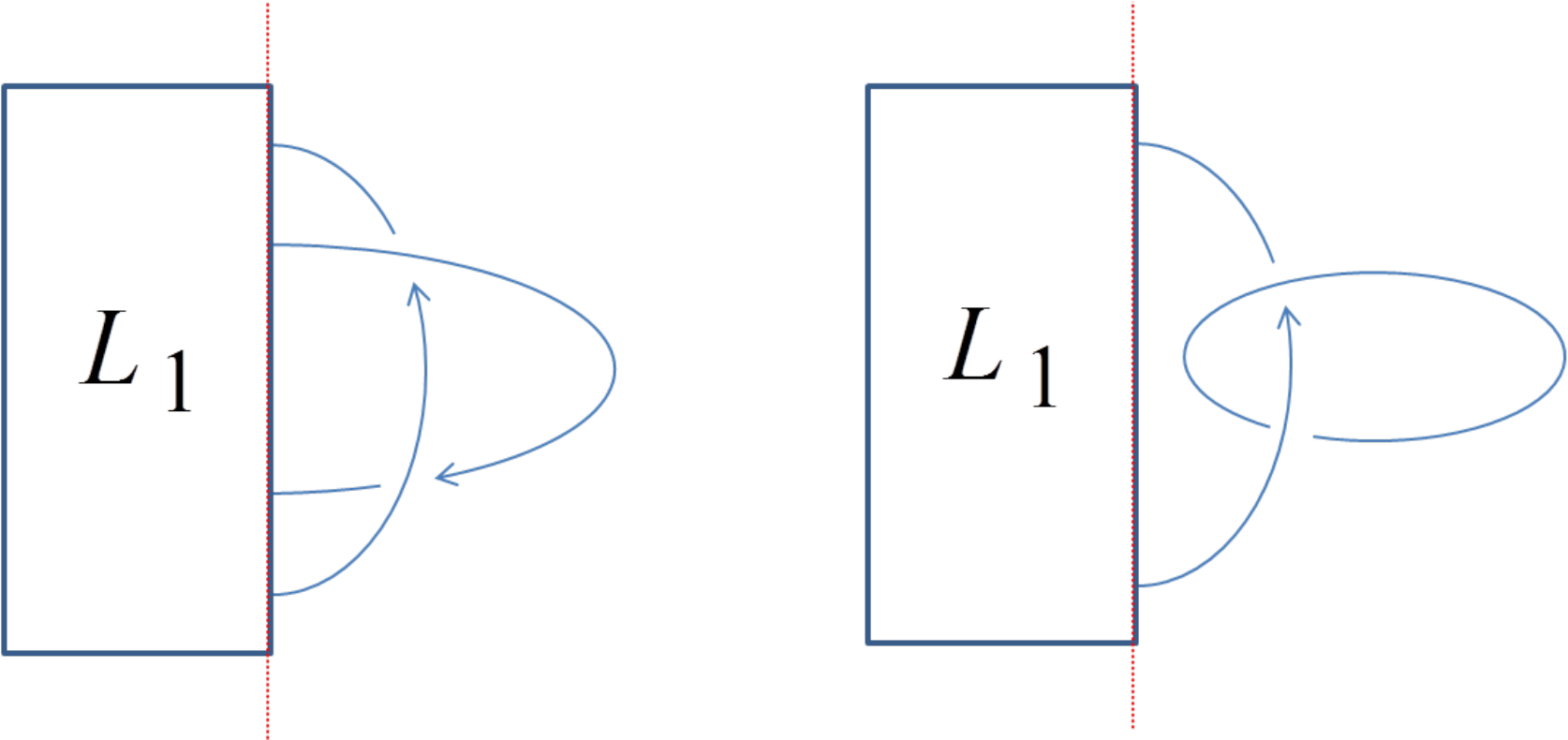}\\
  \end{center}
  \caption{Left: Half Solomon link diagram. Right: Hopf connected sum link diagram. The dashed line represent the admissible cut.}\label{Half_Solomon}
\end{figure}

\begin{cor}
Consider one of the link diagrams in Fig. \ref{Half_Solomon} or its mirror image. Then, the spectral sequence of its double Khovanov complex collapse.
\end{cor}
\begin{proof}
First, consider the Hopf connected sum link diagram. Without loss of generality, suppose its $full$-surgery $L_{2}$ is the Hopf link with positive and negative crossing numbers $(l_{+}^{2}, l_{-}^{2})=(2,0)$. The respective Khovanov homology is:
$$Kh(\textsl{Hopf\ link})\cong V\{1\}\oplus V\{5\}[2]$$
By formula \eqref{spectral_sequence}, the second page has non zero entries in the zeroth and second rows only and because of the fact that the differential at the second page has $(2,1)$-grade, the spectral sequence collapse.

Now, consider the Half Solomon link diagram. The corresponding surgeries are $L_{2}^{full}\simeq unknot$ and $L_{2}^{full}\simeq \textsl{Hopf\ link}$ such that the resulting Hopf link diagram is the one just considered. Because $Kh(\textsl{unknot})\cong V$, the second page has non zero entries in the zeroth and second rows only as before and the spectral sequence collapse.

The mirror image case is completely similar and we have the result.
\end{proof}

\subsection{Example}

Consider the Solomon link diagram $L$ such that $(l_{+},l_{-})=(4,0)$ and an admissible cut $C$. The $full$-surgeries are equal and the resulting knot is ambient isotopic to the unknot. The $trivial$-surgeries are also equal and coincide with the Hopf link diagram such that $(l_{+}, l_{-})=(2,0)$:
\begin{eqnarray*}
L_{1}^{full} &\simeq& L_{2}^{full}\ \ \ \simeq \textsl{unknot} \\
L_{1}^{trivial} &\simeq& L_{2}^{trivial}= \textsl{Hopf\ link}\ \ \ \ \ \ \ \ (l_{+}, l_{-})=(2,0)
\end{eqnarray*}

The Khovanov complex $\mathcal{C}(L_{1}^{full})$ is the following:

$$\xymatrix @R=1pc {	&	&	V^{\otimes 2}\{3\} \ar[rrd]^{m}	&	& 	\\
V^{\otimes 3}\{2\}	\ar[rru]^{id_{2}\otimes m_{13}} \ar[rrd]_{m_{12}\otimes id_{3}} &	&	\bigoplus			&	& V\{4\} \\
					&	&	V^{\otimes 2}\{3\} \ar[rru]_{-m}	&	& \\  
\save 	(18, -10)   *=(55,30){}*\frm{.},
		(63, -10)   *=(20,30){}*\frm{.},
		(18, -30)   *=(55,10){\mathcal{C}_{1}(\alpha)=full}*\frm{},		
		(63, -30)   *=(20,10){\mathcal{C}_{1}(\beta)=trivial}*\frm{},		
\restore }$$
The right box encloses those terms whose associated states $\alpha$ verify $\mathcal{C}_{1}(\alpha)=full$ while the second box encloses the term whose associated $\beta$ state verify $\mathcal{C}_{1}(\beta)=trivial$. By definition \ref{Bigraded_vector_space}, we have the bigraded vector spaces $\mathcal{C}(L_{1})_{full}$ and $\mathcal{C}(L_{1})_{trivial}$:

$$\xymatrix @R=1pc {V\{2\}	& \bigoplus	&	2\ k\{3\}[1]	& \bigoplus	& k\{4\}[2]	\\
\save 	(18, 0)   *=(55,10){}*\frm{.},
		(73, 0)   *=(20,10){}*\frm{.},
		(18, -10)   *=(55,10){\mathcal{C}(L_{1})_{full}}*\frm{},		
		(73, -10)   *=(20,10){\mathcal{C}(L_{1})_{trivial}}*\frm{},		
\restore }$$

The Khovanov complex $\mathcal{C}(L_{1}^{trivial})$ is the following:

$$\xymatrix @R=1pc {	&	&	V\{3\} \ar[rrd]^{\Delta}	&	& 	\\
V^{\otimes 2}\{2\}	\ar[rru]^{m} \ar[rrd]_{m} &	&	\bigoplus			&	& V^{\otimes 2}\{4\} \\
					&	&	V\{3\} \ar[rru]_{-\Delta}	&	& }$$

By Proposition, the double Khovanov complex $\mathcal{CC}(L;C)$ is the following:

$$\xymatrix{
V^{\otimes 2}\{6\} \ar[rrr]^{d_{0}^{trivial}\otimes k\{4\}} & & & 2.V\{7\} \ar[rrr]^{d_{1}^{trivial}\otimes k\{4\}} & & & V^{\otimes 2}\{8\} \\
& & & & & & \\
2.V^{\otimes 3}\{5\} \ar[rrr]^{d_{0}^{full}\otimes 2.k\{3\}} \ar[uu]_{V\{2\}\otimes d_{1}^{full}} & & & 4.V^{\otimes 2}\{6\} \ar[rrr]^{d_{1}^{full}\otimes 2.k\{3\}} \ar[uu]_{2.k\{3\}\otimes d_{1}^{full}} & & & 2.V\{7\} \ar[uu]_{k\{4\}\otimes d_{1}^{trivial}} \\
& & & & & & \\
V^{\otimes 4}\{4\} \ar[rrr]^{d_{0}^{full}\otimes V\{2\}} \ar[uu]_{V\{2\}\otimes d_{0}^{full}} & & & 2.V^{\otimes 3}\{5\} \ar[rrr]^{d_{1}^{full}\otimes V\{2\}} \ar[uu]_{2.k\{3\}\otimes d_{0}^{full}} & & & V^{\otimes 2}\{6\} \ar[uu]_{k\{4\}\otimes d_{0}^{trivial}} }$$
This double complex has the property that:
$$\mathcal{C}(L)\simeq Tot\left(\mathcal{CC}(L;C)\right)$$
hence to calculate the Khovanov homology of $L$ we need to calculate the spectral sequence of the double complex $\mathcal{CC}(L;C)$. The Khovanov homology of the surgeries reads as follows:
$$H(\mathcal{C}(L_{1}^{full})= Kh(L_{1}^{full})\cong Kh(\textsl{unknot})\cong V$$
$$H(\mathcal{C}(L_{1}^{trivial})= Kh(\textsl{Hopf\ link})\cong V\{1\}\oplus V\{5\}[2]$$

By direct calculation we have the following chain complex isomorphism:

$$\xymatrix{
\zeta: & V\{2\}\otimes Kh^{0}\left(L_{2}^{full}\right) \ar[rr]^{[d_{0}^{full}\otimes V\{2\}]} \ar[d]^{\cong}   & & 2.Kh^{0}\left(L_{2}^{full}\right)\{3\} \ar[rr]^{[d_{1}^{full}\otimes V\{2\}]} \ar[d]^{\cong}  & & Kh^{0}\left(L_{2}^{trivial}\right)\{4\} \ar[d]^{\cong} \\
\zeta': & V^{\otimes 2}\{2\} \ar[rr]^{\left(
\begin{array}{c}
m\\
m\\
\end{array}
\right)}   & & 2.V\{3\} \ar[rr]^{(m\circ\Delta, -m\circ\Delta)}   & & V\{5\}  }$$
hence their cohomology is the following:
\begin{equation}\label{cohomology_calculation}
H(\zeta)\cong H(\zeta')\cong V\{1\}\oplus k\{2\}[1] \oplus k\{6\}[2]
\end{equation}

The first page $E^{1}$ of the spectral sequence is the following:

$$\xymatrix{
0 \ar[rr] & &  0 \ar[rr]  & & V\{9\} \\
& & & & \\
0 \ar[rr]   & & 0 \ar[rr]   & & 0  \\
& & & &  \\
V^{\otimes 2}\{2\} \ar[rr]^{\left(
\begin{array}{c}
m\\
m\\
\end{array}
\right)} \ar[]-<3.5em,1em>;[uuuu]-<3.5em,-1em> \ar[]-<3.5em,1em>;[rrrr]-<-3.5em,1em> &  & 2.V\{3\} \ar[rr]^{(m\circ\Delta, -m\circ\Delta)}  &  & V\{5\} \\
\save 	(-20, -22)   *=(10,10){E^{1}=}*\frm{}
\restore }$$

\noindent By the isomorphism \ref{cohomology_calculation}, the second page $E^{2}$ of the spectral sequence reads:

$$\xymatrix{
0  & &  0  &  & V\{9\} \\
& & & & \\
0  & & 0  & & 0  \\
& & & &  \\
V\{1\} \ar[]-<3.5em,1em>;[uuuu]-<3.5em,-1em> \ar[]-<3.5em,1em>;[rrrr]-<-3.5em,1em> & & k\{2\}  & & k\{6\} \\
\save 	(-20, -22)   *=(10,10){E^{2}=}*\frm{}
\restore }$$

\noindent The spectral sequence collapse at the second page and we finally have the Khovanov homology of $L$:
$$Kh(L)\cong V\{1\}\oplus k\{2\}[1] \oplus k\{6\}[2] \oplus V\{9\}[4]$$

\appendix
\section{Jones polynomial splitting formula}

Recall the relation between the Jones polynomial in the Khovanov parameter $q$ and the Khovanov polynomial:
$$J(L)(q)= (q+q^{-1})^{-1}Kh(L)(-1,q)= (q+q^{-1})^{-1}\chi_{q}(\mathcal{C}(L))$$
where $\chi_{q}$ is the $q$-graded Euler characteristic. The Khovanov parameter $q$ is related to the usual Jones parameter $t$ by the relation: $q= -t^{1/2}$.

As an immediate corollary of equation \eqref{First_Groth_eq}, we have:

\begin{cor}\label{Jones_Splitting}
Given an admissible cut $C$ of an oriented link diagram $L$, there is a splitting formula for the Jones polynomial respect to the Khovanov parameter in the field of rational functions $k(q)$:
\begin{equation}\label{Jones_Splitting_Formula}
J(L)(q) = \sum_{\mathcal{A},\mathcal{B}\in NC_{n}} J\left( L_{1}^{\mathcal{A}}\right)(q)\ b_{\mathcal{A}\mathcal{B}}(q)\ J\left( L_{2}^{\mathcal{B}}\right)(q)
\end{equation}
where the splitting matrix $b(q)$ is the inverse of the matrix:
$$c^{\mathcal{A}\mathcal{B}}= (q+q^{-1})^{n+|\mathcal{A}\vee\mathcal{B}|-|\mathcal{A}\wedge\mathcal{B}|-1}$$
\end{cor}

Recall that the matrix $(c^{\mathcal{A}\mathcal{B}}(q))$ is invertible for its diagonal is the higher degree term and cannot be canceled by any other term. In particular, the analog of corollary \ref{Thevenin} holds for the Jones polynomial:

\begin{cor}
Given an admissible cut $C$ of an oriented link diagram $L$, if there is a tangle $T_{2}$ such that the Jones polynomial of its closures equals the Jones polynomial of the surgeries $L_{2}^{\mathcal{A}}$ respectively; i.e. $J(T_{2}^{\mathcal{A}})= J(L_{2}^{\mathcal{A}})$, then:
$$J(L)= J\left( D_{1}( T_{2})\right)$$
where $D$ is the corresponding planar arc operation of the tangle $T_{1}$ obtained from the cut $C$.
\end{cor}

\begin{figure}
\begin{center}
  \includegraphics[width=0.5\textwidth]{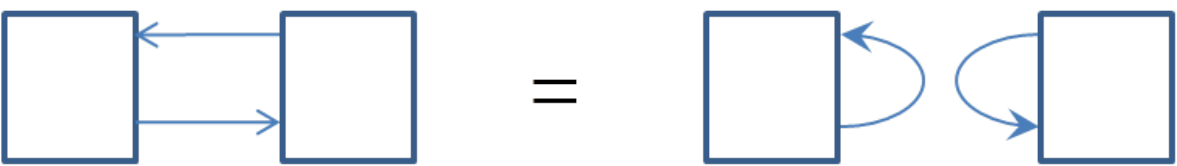}\\
  \end{center}
  \caption{Schematic picture of the Jones Polynomial factorization formula \eqref{factorization}.}\label{Jones_Splitting_I}
\end{figure}

In the case $n$ equals one, the formula \eqref{Jones_Splitting_Formula} reproduce the well known factorization of a connected sum:
\begin{equation}\label{factorization}
J(L_{1}\# L_{2})=J(L_{1})J(L_{2})
\end{equation}
Formula \eqref{factorization} is illustrated in Fig. \ref{Jones_Splitting_I}.

\begin{figure}
\begin{center}
  \includegraphics[width=1.0\textwidth]{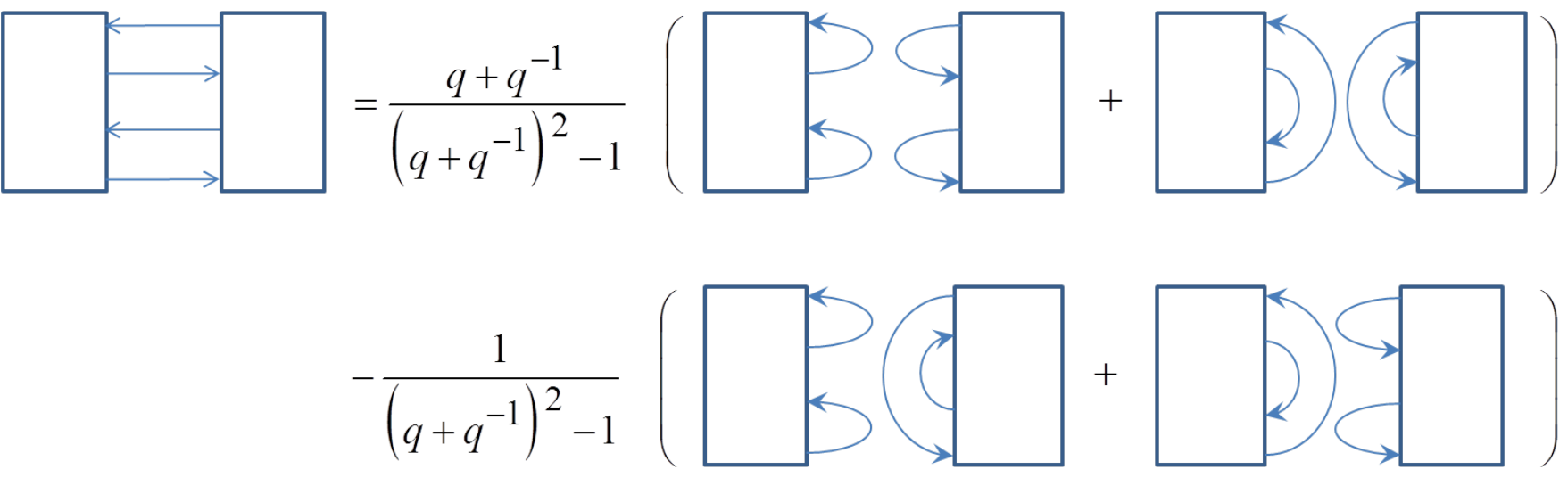}\\
  \end{center}
  \caption{Schematic picture of the Jones Polynomial splitting formula \eqref{SplittingII}.}\label{Jones_Splitting_II}
\end{figure}

In the case $n$ equals two, the splitting formula \eqref{Jones_Splitting_Formula} reads as follows:
\begin{eqnarray}\label{SplittingII}
J(L) &=& \frac{q+q^{-1}}{\left(q+q^{-1}\right)^{2}-1}\left(J(L_{1})J(L_{2})+J(\widehat{L_{1}})J(\widehat{L_{2}})\right) \\
&-& \frac{1}{\left(q+q^{-1}\right)^{2}-1}\left(J(L_{1})J(\widehat{L_{2}})+J(\widehat{L_{1}})J(L_{2})\right) \nonumber
\end{eqnarray}
where $\widehat{L_{i}}$ denotes the trivial surgery. This formula is illustrated in Fig. \ref{Jones_Splitting_II}.

\begin{figure}
\begin{flushleft}
  \includegraphics[width=1.1\textwidth]{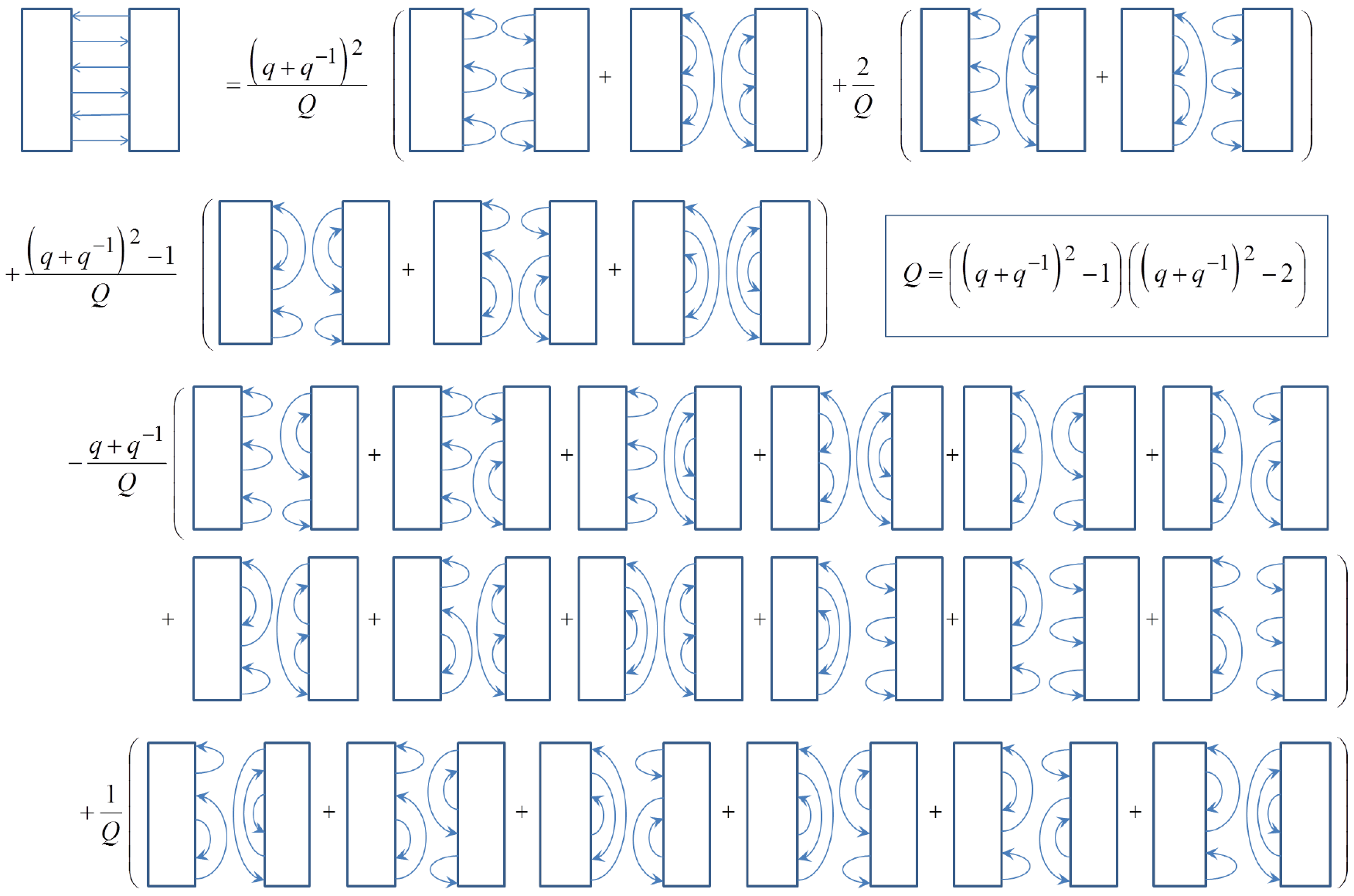}\\
  \end{flushleft}
  \caption{Schematic picture of the Jones Polynomial splitting formula.}\label{Jones_Splitting_III}
\end{figure}

In the case $n$ equals three, the splitting formula \eqref{Jones_Splitting_Formula} is illustrated in Fig. \ref{Jones_Splitting_III}.

\section*{Acknowledgments}
The author is grateful to Consejo Nacional de Ciencia y Tecnolog\'ia (CONACYT), for its \textit{C\'atedras CONACYT} program.


\begin{thebibliography}{0}




\bibitem[1]{classicBN}
D.Bar-Natan, \emph{On Khovanov's categorification of the Jones polynomial}, Algebraic and Geometric Topology 2 (2002), 337-370. 


\bibitem[2]{Tangles_BN}
D.Bar-Natan, \emph{Khovanov's Homology for Tangles and Cobordisms}, Geometry and Topology 9-33 (2005) 1443-1499.


\bibitem[3]{Fast_Khovanov_BN}
D.Bar-Natan, \emph{Fast Khovanov Homology Computations}, Journal of Knot Theory and Its Ramifications, Vol 16, 3.


\bibitem[4]{Khovanov_original}
M.Khovanov, \emph{A categorification of the Jones polynomial}, Duke Math. J. 101 (2000), 359-426.

\bibitem[5]{Kreweras}
G.Kreweras, \emph{Sur les partitions non croisées d'un cycle}, Discrete Math., 1 (1972), 333-350.

\bibitem[6]{Lee}
E.S.Lee, \emph{An endomorphism of the Khovanov invariant}, Advances in Mathematics, Vol 197, 2, 554-586.

\bibitem[7]{open_closed_TQFT}
A.D.Lauda, H.Pfeiffer, \emph{Open-closed TQFTs extend Khovanov homology from links to tangles}, Journal of Knot Theory and Its Ramifications, Vol 18, 1.


\bibitem[8]{mccammond}
J.Mccammond, \emph{Noncrossing partitions in surprising locations}, Amer. Math. Monthly, 113 (2006), 598-610.


\end{thebibliography}
\end{document}